\documentclass[11pt]{article}
\date{}
\usepackage{amssymb}
\usepackage{amsmath}
\usepackage{amsthm}
\usepackage{graphicx}
\usepackage{indentfirst}
\allowdisplaybreaks[4]
\newtheorem{theorem}{Theorem}

\newtheorem{corollary}[theorem]{Corollary}

\newtheorem{definition}[theorem]{Definition}

\newtheorem{lemma}[theorem]{Lemma}

\newtheorem{remark}[theorem]{Remark}

\numberwithin{equation}{section}
\numberwithin{theorem}{section}
\newcommand{\keywords}[1]{\par\noindent\textbf{Keywords:} #1}
\newcommand{\subjclass}[2]{
  \par\noindent\textbf{Mathematics Subject Classification} #2}

\begin{document}

\title{On nonnegative solutions of the parabolic differential inequality with $(p,q)$-Laplace on Riemannian manifolds }

\author{Biqiang Zhao\\
\small Beijing International Center for Mathematical Research,
\\
\small Peking University, Beijing, China\\ \small 2306394354@pku.edu.cn}
            
\maketitle
\setlength{\parindent}{2em}

\begin{abstract}
In this paper, we establish Liouville-type theorems for parabolic differential inequalities with $(p,q)-$Laplacian operator on Riemannian manifolds. By a test function argument, we establish nonexistence results under suitable weighted volume growth assumptions involving potential. In particular, we can obtain nonexistence results for a wider class of parabolic inequalities.

\end{abstract}

\subjclass :{ 35K59, 35K92, 53C20, 35R01 }


\section{Introduction}
\label{1}
 In this paper, we are concerned with the parabolic quasilinear inequalities
\begin{align}  
    \begin{cases}
        \partial_tu-\Delta_pu-\Delta_qu\geq V(x,t)u^s & in\quad   M\times (0,\infty) 
           \\
           u=u_0\geq 0  & in \quad  M\times \{0\}, \\
    \end{cases} 
\end{align}
where $(M,g)$ is a complete, non-compact, Riemannian manifold with the metric $g$ and $\Delta_z u=\mathrm{div}(|\nabla u|^{z-2}\nabla u),\ z\in \{p,q\}$. Moreover, we assume that $1<q\leq p$, $s>\mathrm{max}\{1,p-1\}$ and $ V(x,t)> 0$ a.e. on $M\times (0,\infty) $.
\par
Liouville theorems for nonlinear parabolic equations has received considerable attention in the literature. In $\mathbb{R}^n$, the following semilinear parabolic Cauchy problem 
\begin{align}  
    \begin{cases}
       \partial_tu-\Delta u= u^s  &in\quad \mathbb{R}^n\times (0,\infty) \\
           u=u_0\geq 0  &in\quad \mathbb{R}^n\times \{0\} ,\\
    \end{cases} 
\end{align}
where $s>1,u_0\in L^\infty(\mathbb{R}^n)$ has been widely studied. Due to the celebrated work of Fujita \cite{Fu}, (1.2) has no global positive solution if $1<s<1+\frac{2}{n}$. On the contrary, for $s>1+\frac{2}{n}$, (1.2) exists a positive solution. In the critical case $ s=1+\frac{2}{n}$, this problem was settled by Hayakawa \cite{Ha} for $n=1,2$ and Kobayashi-Sirao-Tanaka \cite{KST} for any natural number. Recently, Quittner \cite{Qu} considered the nonlinear heat equation
\begin{align*}
    \partial_tu-\Delta u= u^s,  \quad in\quad \mathbb{R}^n\times (-\infty,\infty)
\end{align*}
and proved that the equation posses no positive solution if 
\begin{align*}
     1<s<\bar{s}=
    \begin{cases}
       \infty,  &if \quad n\leq 2, \\
          \frac{n+2}{n-2},   &if \quad n> 2.\\
    \end{cases} 
\end{align*}
This result corresponds to the stationary solution in \cite{GS}.
\par
Liouville type results can also be generalized to parabolic inequalities, for instance, the inequalities 
\begin{align}  
    \begin{cases}
       \partial_tu-\Delta_pu\geq u^s  &in\quad \mathbb{R}^n\times (0,\infty) \\
           u=u_0\geq 0  &in\quad \mathbb{R}^n\times \{0\} .\\
    \end{cases} 
\end{align}
 Mitidieri and Pohozaev \cite{MP} proved that (1.3) does not admit nontrivial nonnegative weak solutions, provided that
\begin{align*}
   p>\frac{2n}{n+1},\quad \mathrm{max}\{p-1,1\}<s\leq p-1+\frac{p}{n}. 
\end{align*}
\par
We can also investigate parabolic equations or inequalities on more general spaces. In \cite{Zh}, Zhang generalized the results of Fujita to noncompact Riemannian manifolds $M^n$ with $n\geq 3$.  Bandle, Pozio and Tesei \cite{BPT}  studied the semilinear Cauchy problem
\begin{align}  
    \begin{cases}
       \partial_tu-\Delta u= h(t)u^s  &in\quad \mathbb{H}^n\times (0,\infty) \\
           u=u_0\geq 0  &in\quad \mathbb{H}^n\times \{0\} .\\
    \end{cases} 
\end{align}
They proved that if $h(t)=1$ or 
\begin{align*}
    \alpha_1 t^v \leq h(t)\leq \alpha_2t^v,\quad t>t_0
\end{align*}
for some $0< \alpha_1\leq \alpha_2, t_0>0 $ and $v>-1$, then (1.4) exists global bounded solutions. For more results, one can refer to \cite{Pu1,Pu2,GSXX,MMG,MMP2,MMP3} and references therein.
\par
Let us mention that nonexistence results of nonnegative nontrivial solutions  to elliptic equations and inequalities by using the upper bound of volume of any geodesic ball. This idea was due to Cheng and Yau \cite{CY}. They showed 
that if $ V_x(B_R)\leq C R^2$ holds for all large enough $R$, then nonnegative superharmonic function on $M$ is identically constant. Here $ V_x(B_R)$ is the volume of geodesic ball of radius $R$ centered at $x$. This idea was developed in \cite{AS,GK} to investigate the differential inequality
\begin{align}
    \mathrm{div}(A(x)\nabla u)+B(x)u^s\leq 0.
\end{align}
More precisely, with $A(x)=B(x)=1$, Grigor’yan and Sun \cite{AS} proved that any nonnegative weak solution of (1.5) is equal to 0, provided that 
\begin{align*}
    V_x(B_R)\leq C R^{\frac{2s}{s-1}}(\mathrm{log}\ R)^{\frac{1}{s-1}}
\end{align*}
for $R$ large enough. The exponents on $R$ and $\mathrm{log}\ R$ are sharp. Otherwise, Grigor’yan and Sun \cite{AS} showed that there exists some model manifold which satisfies the volume growth condition and admits positive solution to (1.5) with $A(x)=B(x)=1$. In \cite{GSV}, Grigor’yan, Sun and Verbitsky obtained the necessary and sufficient condition for the existence of positive solutions in terms of Green function of $\Delta$.  Later, Mastrolia, Monticelli and Punzo \cite{MMP1} studied the existence of positive solutions to the class of elliptic differential inequalities
\begin{align}
    \frac{1}{a(x)}\mathrm{div}(a(x)|\nabla u|^{p-2}\nabla u)+V(x)u^s\leq 0
\end{align}   
and showed that the potential function gives a direct influence on the volume
growth. For parabolic differential inequalities, Mastrolia, Monticelli and Punzo \cite{MMP2} investigated a class quasilinear parabolic inequalities
with a potential on Riemannian manifolds. In \cite{MMP3}, Mastrolia, Monticelli and Punzo obtained Liouville type theorems for parabolic inequalities with a potential in bounded domains.
\par
In the last few years, the $(p,q)$-Laplace operator attracts a lot of attention. The idea of studying such operators comes from the problems of the calculus of variations and nonlinear elasticity theory, cf. \cite{Ma1,Ma2,Zh1,Zh2}.  In \cite{BT}, Bobkov and Tanaka studied the existence and non-existence of positive solutions for the $(p,q)$-Laplace equations with two parameters in a bounded domain. Wang and Zhang \cite{WZ} obtained the gradient estimates for solutions to nonlinear elliptic
equation driven by the $(p,q)$-Laplace operator. Recently, Bhakta, Biswas and Filippucci \cite{BBF} established several Liouville-type theorems for differential quasilinear inequalities with $(p,q)$-Lapalce in the entire $\mathbb{R}^n$ (or an exterior domain). In \cite{BBF2}, Bhakta, Biswas and Filippucci obtained Liouville theorems for ($p,q)$-Laplace elliptic equations with source terms involving gradient nonlinearity. In \cite{Zhao}, the author studied the nonnegative solutions of the differential inequality with
$(p,q)$-Laplacian operator on Riemannian manifolds.
\par
The aim of this paper is to obtain Liouville-type theorems for quasilinear inequalities (1.1) on Riemannian manifolds.
\par
Throughout the paper, we assume that $1<q\leq p$ and $r(x)$ is the Riamannian distance from $x$ to a fixed point $x_0\in M$. Denote by $B_R$ the ball centered at $x_0$ with radius $R$. Since the constant $C$ is not important, it may vary at different occurrences.
\par
Before starting with the main theorems, we give the following definitions.

\begin{definition}
    Let $p\geq q>1, s>p-1$, $V>0$ a.e. on $M\times [0,\infty)$ and $V\in L^1(M\times [0,\infty))$. For $R>0$ and $\theta_1,\theta_2\geq 1$, let $S=M\times [0,\infty)$ and 
    \begin{align*}
        E_R=\{(x,t)\in S : t^{\theta_1}+r^{\theta_2}(x)\leq R^{\theta_2}\}.
    \end{align*}
    For $z\in \{p,q\}$, we define 
    \begin{align*}
        \bar{s}_1&=\frac{s}{s-1}\theta_2, & \bar{s}_2&=\frac{1}{s-1},
        \\
        \bar{s}_{3,z}&=\frac{sz}{s-z+1}\theta_2, & \bar{s}_{4,z}&=\frac{z-1}{s-z+1}.
    \end{align*}
    We introduce the following volume growth conditions HP1 and HP2, which we will derive our nonexistence results.
    \par
    HP1: Assume that: (1) there exist $\theta_1,\theta_2\geq1$, $C_0,C>0$ and $s_2\in [0,\bar{s}_2)$ such that for every $R>0$ large enough and every $\epsilon>0$ small enough
    \begin{align}
        \int\int_{E_{2^{1/\theta_2}R}\setminus  E_{R}} t^{(\theta_1-1)(\frac{s}{s-1}-\epsilon)} V^{-\bar{s}_2+\epsilon}d\mu dt \leq CR^{\bar{s}_1+C_0\epsilon} (\mathrm{log}\ R)^{s_2};
    \end{align}
    (2) for $z\in \{p,q\}$, there exist $\theta_1,\theta_2\geq1$, $C_0,C>0$ and $s_{4,z}\in [0,\bar{s}_{4,z})$ such that for every $R>0$ large enough and every $\epsilon>0$ small enough
    \begin{align}
        \int\int_{E_{2^{1/\theta_2}R}\setminus  E_{R}} r(x)^{(\theta_2-1)z(\frac{s}{s-z+1}-\epsilon)} V^{-\bar{s}_{4,z}+\epsilon}d\mu dt \leq CR^{\bar{s}_{3,z}+C_0\epsilon} (\mathrm{log}\ R)^{s_{4,z}};
    \end{align}
    \par
    HP2: Assume that: (1) there exist $\theta_1,\theta_2\geq1$ and $C_0,C>0$ such that for every $R>0$ large enough and every $\epsilon>0$ small enough
    \begin{align}
        \int\int_{E_{2^{1/\theta_2}R}\setminus  E_{R}} t^{(\theta_1-1)(\frac{s}{s-1}-\epsilon)} V^{-\bar{s}_2+\epsilon}d\mu dt \leq CR^{\bar{s}_1+C_0\epsilon} (\mathrm{log}\ R)^{\bar{s}_2}.
    \end{align}
    (2)  for $z\in \{p,q\}$, there exist $\theta_1,\theta_2\geq1$ and $C_0,C>0$ such that for every $R>0$ large enough and every $\epsilon>0$ small enough
    \begin{align}
        &\int\int_{E_{2^{1/\theta_2}R}\setminus  E_{R}} r(x)^{(\theta_2-1)z(\frac{s}{s-z+1}-\epsilon)} V^{-\bar{s}_{4,z}+\epsilon}d\mu dt \leq CR^{\bar{s}_{3,z}+C_0\epsilon} (\mathrm{log}\ R)^{\bar{s}_{4,z}},
        \\
        &\int\int_{E_{2^{1/\theta_2}R}\setminus  E_{R}} r(x)^{(\theta_2-1)z(\frac{s}{s-z+1}+\epsilon)} V^{-\bar{s}_{4,z}-\epsilon}d\mu dt \leq CR^{\bar{s}_{3,z}+C_0\epsilon} (\mathrm{log}\ R)^{\bar{s}_{4,z}}.
    \end{align}
\end{definition}
\begin{remark}
    By Fatou’s Lemma, the above conditions hold for $\epsilon=0$. Hence, we do not need the condition (1.9) in \cite{MMP2}. This has been shown in \cite{VGM}.
\end{remark}
\begin{remark}
    Note that $t\leq R^{\frac{\theta_2}{\theta_1}} $ and $r(x)\leq R$ on $E_R$. Hence the HP1 holds if $V$ satisfies the following integral estimates for $R>0$ large enough and $\epsilon>0$ small enough:
    \begin{align*}
        &\int\int_{E_{2^{1/\theta_2}R}\setminus  E_{R}}  V^{-\bar{s}_2+\epsilon}d\mu dt \leq CR^{\frac{\bar{s}_1}{\theta_1}+C_0\epsilon} (\mathrm{log}\ R)^{s_2},
        \\
        &\int\int_{E_{2^{1/\theta_2}R}\setminus  E_{R}}  V^{-\bar{s}_{4,z}+\epsilon}d\mu dt \leq CR^{\frac{\bar{s}_{3,z}}{\theta_2}+C_0\epsilon} (\mathrm{log}\ R)^{s_{4,z}},
    \end{align*}
    where $z\in \{p,q\}, s_2\in [0,\bar{s}_2) $ and $ s_{4,z}\in [0,\bar{s}_{4,z})$. Similar results hold for HP2.
\end{remark}
     \begin{remark}
         Analogous to quasilinear case \cite{Zhao}, the hypotheses highlight the dominant influence of the lower order term, that is $q$-Laplacian operator, which primarily affects the structure of the analysis. For example, if $V=1$, we can obtain the nonexistence result for the condition holds for $z=q$.
     \end{remark}
    Now we give the main theorem of this paper (the notion of weak solutions is in Section 2).
    \begin{theorem}
        Let $p\geq q>1,\ s>\mathrm{max}\{p-1,1\}$ and $ V\in L^1(M\times[0,\infty))$ with $V> 0$ a.e. on $M\times[0,\infty)$. Let $u\in W^{1,p}_{loc}(M\times[0,\infty))\cap L^s_{loc}(M\times[0,\infty))$ be a nonnegative weak solution of (1.1) and $u_0\in L^1(M),\ u_0\geq 0$ a.e on $M$, then $u\equiv0$ a.e on $M\times[0,\infty)$ provided that HP1 holds.
    \end{theorem}
    
\begin{theorem}
           Let $p\geq q>1,\ s>\mathrm{max}\{p-1,1\}$ and $ V\in L^1(M\times[0,\infty))$ with $V> 0$ a.e. on $M\times[0,\infty)$. Let $u\in W^{1,p}_{loc}(M\times[0,\infty))\cap L^s_{loc}(M\times[0,\infty))$ be a nonnegative weak solution of (1.1) and $u_0\in L^1(M),\ u_0\geq 0$ a.e on $M$, then $u\equiv0$ a.e on $M\times[0,\infty)$ provided that HP2 holds. 
       \end{theorem}
    Obviously, Theorem 1.5 and Theorem 1.6 generalize the result in \cite{MMP2} to $(p,q)$-Laplace and are the parabolic counterparts of the results in \cite{Zhao}. As in \cite{MMP2}, we have the following corollaries.
    \begin{corollary}
         Let $M=\mathbb{R}^n$ with the standard metric. Assume that $p\geq q>1,\ s>\mathrm{max}\{p-1,1\}$, $V=1$ and 
         \begin{align*}
             s\leq \frac{(n+1)}{n}q-1,
         \end{align*}
         then any nonnegative weak solution of (1.1) is equal to 0 a.e on $S$.         
         \end{corollary}
         \begin{corollary}
             Let $M$ be a complete noncompact Riemannian manifold, $p\geq q>1,\ s>\mathrm{max}\{p-1,1\}$ and $  u_0\in L^1(M\times[0,\infty)), u_0\geq 0 $ a.e on $M$. Assume that $V\in L^1(M\times[0,\infty))$ satisfies that
             \begin{align*}
                 V(x,t)\geq f(t)h(x),\quad  for\ a.e.\  (x,t)\in M\times[0,\infty).
             \end{align*}
             Here $f:(0,\infty)\xrightarrow{}\mathbb{R}$ and $h:M\xrightarrow{} \mathbb{R}$ satisfy
             \begin{align}
                 &0<f(t)\leq C(1+t)^\alpha, \quad for\ a.e.\ t\in (0,\infty),
                 \\
                 &0<h(x)\leq C(1+r(x))^\beta, \quad for\ a.e.\ x\in M,
             \end{align}
             and for $z\in\{p,q\}$
             \begin{align}
                &\int_{B_R}h(x)^{-\frac{1}{s-1}}d\mu \leq CR^{\sigma_1}(\mathrm{log}R)^{\delta_1},\quad  \int_0^Tf(t)^{-\frac{1}{s-1}}dt\leq CT^{\sigma_2}(\mathrm{log}T)^{\delta_2},
                \\
                &\int_{B_R}h(x)^{-\frac{z-1}{s-z+1}}d\mu \leq CR^{\sigma_{3,z}}(\mathrm{log}R)^{\delta_{3,z}},\quad  \int_0^Tf(t)^{-\frac{s-1}{s-z+1}}dt\leq CT^{\sigma_{4,z}}(\mathrm{log}T)^{\delta_{4,z}} ,
             \end{align}
             where $T,R$ large enough and $\alpha,\beta,\sigma_1,\sigma_2,\sigma_{3,z},\sigma_{4,z},\delta_1,\delta_2,\delta_{3,z},\delta_{4,z} \geq 0$. Assume that the following conditions            
             ~\\
             (1) $\delta_1+\delta_2<\frac{1}{s-1}$, $ \delta_{3,z}+\delta_{4,z}<\frac{z-1}{s-z+1}$;
             \\
             (2) $0\leq \sigma_2\leq \frac{s}{s-1}$, $ 0\leq\sigma_{3,z}\leq \frac{sz}{s-z+1}$;
             \\
             (3) If $\sigma_2=\frac{s}{s-1}$, then $\sigma_1=0$; if $\sigma
             _{3,z}=\frac{sz}{s-z+1}$, then $\sigma_{4,z}=0$;
             \\
             (4) $\sigma_1\sigma_{4,z}\leq (\frac{s}{s-1}-\sigma_2)(\frac{sz}{s-z+1}-\sigma_{3,z})$
             \\
             hold for $z\in\{p,q\}$. Then (1.1) does not admit any nontrivial nonnegative weak solution.
         \end{corollary}
         \begin{corollary}
             Assume that $V$ satisfies the condition in Corollary 1.7. In addition, $f(t),h(x)$ have a lower bound respectively, i.e.,
             \begin{align}
                 &C^{-1}(1+t)^{-\alpha}\leq f(t)\leq C(1+t)^\alpha, \quad t\in (0,\infty),
                 \\
                 & C^{-1}(1+r(x))^{-\beta}\leq h(x)\leq C(1+r(x))^\beta, \quad x\in M.
             \end{align}
             Assume that the following conditions            
             ~\\
             (1) $\delta_1+\delta_2\leq\frac{1}{s-1}$, $ \delta_{3,z}+\delta_{4,z}\leq \frac{z-1}{s-z+1}$;
             \\
             (2) $0\leq \sigma_2\leq \frac{s}{s-1}$, $ 0\leq\sigma_{3,z}\leq \frac{sz}{s-z+1}$;
             \\
             (3) If $\sigma_2=\frac{s}{s-1}$, then $\sigma_1=0$; if $\sigma
             _{3,z}=\frac{sz}{s-z+1}$, then $\sigma_{4,z}=0$;
             \\
             (4) $\sigma_1\sigma_{4,z}\leq (\frac{s}{s-1}-\sigma_2)(\frac{sz}{s-z+1}-\sigma_{3,z})$
             \\
             hold for $z\in\{p,q\}$. Then (1.1) does not admit any nontrivial nonnegative weak solution.
         \end{corollary}
         
         Combining the result in \cite{VGM} and the proof of Theorem 1.5 and Theorem 1.6, we can actually establish nonexistence results for a wider class of inequalities of the type
         \begin{align*}
                  \partial_tu\geq  \mathrm{div}(|\nabla u|^{q-2}f(|\nabla u|)\nabla u)+V(x,t)u^s,\quad in\  M\times (0,\infty).
               \end{align*}
               Here $f$ satisfies the condition that there exists constants $K\geq 0,b_i\geq a_i>0$ and $v_i>0$ such that
               \begin{align*}
                   \sum\limits_{i=1}^{k} a_it^{v_i}\leq f(t) \leq \sum\limits_{i=1}^{k} b_it^{v_i}+K, \quad \forall t\geq 0.
               \end{align*}
               \par
             The rest of the paper is organized as follows. In Section 2, we introduce some preliminary results, which will be used in the proof. In Section 3 and Section 4, we give the proof of Theorem 1.5 and Theorem 1.6 respectively. Section 5 deals with the proof of corollaries.

             \section{Preliminary}
             We begin with the definition of the weak solution of (1.1).
             \begin{definition}
               Let $p\geq q>1,\ s>p-1$, $ V(x,t)\in L^1(S)$ with $V> 0$ a.e. on $M\times[0,\infty)$ and $u_0 \geq  0$ a.e. on $M$. We say that $u\in W^{1,p}_{loc}(M\times[0,\infty))\cap L^s_{loc}(M\times[0,\infty) )$ is a nonnegative weak solution of (1.1) if $u\geq 0$ a.e. on $M\times[0,\infty)$ and for every $0\leq \psi\in W^{1,p}(M\times[0,\infty))\cap L^{\infty}(M\times[0,\infty))$ with compact support, one has
               \begin{align}
                   &\int_0^\infty\int_M u\partial_t \psi d\mu dt+\int_Mu_0\psi(x,0)d\mu+\int_0^\infty\int_MVu^s\psi d\mu 
                   \nonumber\\
                   \leq& \int_0^\infty\int_M |\nabla u|^{p-2}\langle\nabla u,\nabla\psi\rangle d\mu dt+\int_0^\infty\int_M |\nabla u|^{q-2}\langle\nabla u,\nabla\psi\rangle d\mu dt.
               \end{align}
           \end{definition}
           Now we state some preliminary results which will be used in the proof.
           \begin{lemma}
               Let $0<a<\frac{1}{2}\mathrm{min}\{q-1,1\}$, $b>\mathrm{max}\{1,\frac{s}{s-1},\frac{ps}{s-p+1}\} $ and $s>\mathrm{max}\{p-1,1\}$. Assume that $u$ is a nonnegative weak solution of (1.1). Then for any $ \varphi\in \mathrm{Lip}(S)$ with compact support and $0\leq \varphi\leq 1$, we have
               \begin{align}
                   &\frac{a}{2}\int_0^\infty \int_M |\nabla u|^pu^{-a-1}\varphi^bd\mu dt+\frac{a}{2}\int_0^\infty \int_M |\nabla u|^qu^{-a-1}\varphi^bd\mu dt
                   \nonumber\\
                   &\qquad +\frac{1}{4}\int_0^\infty\int_Mu^{-a+s}\varphi^bVd\mu dt
                   \nonumber\\
                   \leq & Ca^{-\frac{(p-1)s}{s-p+1}}\int_0^\infty\int_M|\nabla \varphi|^{\frac{s-a}{s-p+1}p}V^{-\frac{p-a-1}{s-p+1}}d\mu dt
                   \nonumber\\
                   &+ Ca^{-\frac{(q-1)s}{s-q+1}}\int_0^\infty\int_M|\nabla \varphi|^{\frac{s-a}{s-q+1}q}V^{-\frac{q-a-1}{s-q+1}}d\mu dt
                   \nonumber\\
                   &+C\int_0^\infty \int_M |\partial_t\varphi|^{\frac{s-a}{s-1}}V^{-\frac{-a+1}{s-1}}d\mu dt.
               \end{align}
           \end{lemma}
           \begin{proof}
               Without loss of generality, we assume that $u$ is positive with 
               $u^{-1}\in L^{\infty}_{loc}(M\times [0,\infty))$. Otherwise, we define $u_\eta=u+\eta$ for any $\eta>0$ and then let $\eta\xrightarrow{} 0$.  Let $\psi=u^{-a}\varphi^b$, then 
               \begin{align}
                   \nabla\psi=-au^{-a-1}\varphi^b\nabla u+bu^{-a}\varphi^{b-1}\nabla  \varphi,\quad \partial_t\psi=-au^{-a-1}\varphi^b\partial_t u+bu^{-a}\varphi^{b-1}\partial_ t\varphi
               \end{align}
              a.e. on $S$. From (2.1), we have
               \begin{align}
                   &\int_0^\infty\int_M u^{s-a}\varphi^bVd\mu dt
                   \nonumber\\
                   \leq &-a\int_0^\infty\int_M |\nabla u|^pu^{-a-1}\varphi^bd\mu dt+b \int_0^\infty\int_M|\nabla u|^{p-2}u^{-a}\varphi^{b-1}\langle \nabla u,\nabla\varphi\rangle d\mu dt
                   \nonumber\\
                   &-a\int_0^\infty\int_M |\nabla u|^qu^{-a-1}\varphi^bd\mu dt+b \int_0^\infty\int_M|\nabla u|^{q-2}u^{-a}\varphi^{b-1}\langle \nabla u,\nabla\varphi\rangle d\mu dt
                   \nonumber\\
                   &+a\int_0^\infty\int_M u^{-a}\varphi^b\partial_t u d\mu dt-b\int_0^\infty\int_M u^{-a+1}\varphi^{b-1}\partial_t\varphi d\mu dt
                   \nonumber\\
                   &-\int_Mu_0^{-a+1}\varphi^b(x,0)d\mu.
               \end{align}
               Since
               \begin{align*}
                   &a\int_0^\infty\int_M u^{-a}\varphi^b\partial_t u d\mu dt=\frac{a}{1-a}\int_0^\infty\int_M \partial_t (u^{1-a})\varphi^b d\mu dt
                   \\
                   =&\frac{a}{1-a}\int_0^\infty\int_M \partial_t(u^{1-a}\varphi^b)-u^{1-a}\partial_t(\varphi^b)d\mu dt
                   \\
                   =&-\frac{a}{1-a}\int_M u_0^{{1-a}}\varphi^b(x,0)d\mu-\frac{ab}{1-a}\int_0^\infty\int_M u^{-a+1}\varphi^{b-1}\partial_t\varphi d\mu dt,
               \end{align*}
               we derive that
               \begin{align}
                   &a\int_0^\infty\int_M u^{-a}\varphi^b\partial_t u d\mu dt-b\int_0^\infty\int_M u^{-a+1}\varphi^{b-1}\varphi_t d\mu dt
                   -\int_Mu_0^{-a+1}\varphi^b(x,0)d\mu
                   \nonumber\\
                   =&-\frac{1}{1-a}\int_M u_0^{{1-a}}\varphi^b(x,0)d\mu-\frac{b}{1-a}\int_0^\infty\int_M u^{-a+1}\varphi^{b-1}\partial_t\varphi d\mu dt.
               \end{align}
               Substituting in (2.4), we have
               \begin{align}
                   &\int_0^\infty\int_M u^{s-a}\varphi^bVd\mu dt
                   \nonumber\\
                   \leq &-a\int_0^\infty\int_M |\nabla u|^pu^{-a-1}\varphi^bd\mu dt+b \int_0^\infty\int_M|\nabla u|^{p-2}u^{-a}\varphi^{b-1}\langle \nabla u,\nabla\varphi\rangle d\mu dt
                   \nonumber\\
                   &-a\int_0^\infty\int_M |\nabla u|^qu^{-a-1}\varphi^bd\mu dt+b \int_0^\infty\int_M|\nabla u|^{q-2}u^{-a}\varphi^{b-1}\langle \nabla u,\nabla\varphi\rangle d\mu dt
                   \nonumber\\
                   &-\frac{1}{1-a}\int_M u_0^{{1-a}}\varphi^b(x,0)d\mu-\frac{b}{1-a}\int_0^\infty\int_M u^{-a+1}\varphi^{b-1}\partial_t \varphi d\mu dt.
               \end{align}
               Using Young's inequality with the pair $(\frac{p-1}{p},\frac{1}{p})$, we obtain
               \begin{align*}
                   &b \int_0^\infty\int_M|\nabla u|^{p-2}u^{-a}\varphi^{b-1}\langle \nabla u,\nabla\varphi\rangle d\mu dt
                   \\
                   \leq& b \int_0^\infty\int_M|\nabla u|^{p-1}u^{-a}\varphi^{b-1}|\nabla\varphi| d\mu dt
                   \\
                   =& \int_0^\infty\int_M\left(\frac{a}{2}\right)^{\frac{p-1}{p}}\varphi^{\frac{p-1}{p}b}u^{-(a+1)\frac{p-1}{p}}|\nabla u|^{p-1} \cdot b \left(\frac{a}{2}\right)^{-\frac{p-1}{p}}\varphi^{\frac{b}{p}-1}u^{1-\frac{a+1}{p}}|\nabla\varphi|d\mu dt
                   \\
                   \leq & \frac{a}{2}\int_0^\infty\int_M |\nabla u|^pu^{-a-1}\varphi^bd\mu dt+b^p\left(\frac{a}{2}\right)^{-(p-1)}\int_0^\infty\int_Mu^{p-a-1}\varphi^{b-p}|\nabla\varphi|^pd\mu dt.
               \end{align*}
               Similarly, we have
               \begin{align*}
                   &b \int_0^\infty\int_M|\nabla u|^{q-2}u^{-a}\varphi^{b-1}\langle \nabla u,\nabla\varphi\rangle d\mu dt
                   \\
                   \leq &\frac{a}{2}\int_0^\infty\int_M |\nabla u|^qu^{-a-1}\varphi^bd\mu dt+b^q\left(\frac{a}{2}\right)^{-(q-1)}\int_0^\infty\int_Mu^{q-a-1}\varphi^{b-q}|\nabla\varphi|^qd\mu dt.
               \end{align*}
               Together with (2.6), we deduce 
               \begin{align}
                    &\frac{a}{2}\int_0^\infty \int_M |\nabla u|^pu^{-a-1}\varphi^bd\mu dt+\frac{a}{2}\int_0^\infty \int_M |\nabla u|^qu^{-a-1}\varphi^bd\mu dt
                   \nonumber\\
                   &\qquad +\int_0^\infty\int_Mu^{-a+s}\varphi^bVd\mu dt
                   \nonumber\\
                   \leq & b^p\left(\frac{a}{2}\right)^{-p+1}\int_0^\infty\int_Mu^{p-a-1}\varphi^{b-p}|\nabla\varphi|^pd\mu dt
                   \nonumber\\
                   &+b^q\left(\frac{a}{2}\right)^{-q+1}\int_0^\infty\int_Mu^{q-a-1}\varphi^{b-q}|\nabla\varphi|^qd\mu dt
                   \nonumber\\
                   &+\frac{b}{1-a}\int_0^\infty\int_M u^{-a+1}\varphi^{b-1}|\partial_t \varphi |d\mu dt.
               \end{align}
               By Young's inequality again, we have
               \begin{align}
                   &b^p\left(\frac{a}{2}\right)^{-p+1}\int_0^\infty\int_Mu^{p-a-1}\varphi^{b-p}|\nabla\varphi|^pd\mu dt
                   \nonumber\\
                   =& b^p\left(\frac{a}{2}\right)^{-p+1}\int_0^\infty\int_M u^{p-a-1}\varphi^{\frac{p-a-1}{s-a}b}V^{\frac{p-a-1}{s-a}}\cdot \varphi^{b\frac{s-p+1}{s-a}-p}V^{-\frac{p-a-1}{s-a}}|\nabla\varphi|^p d\mu dt
                   \nonumber\\
                   \leq & \frac{1}{4}\int_0^\infty\int_Mu^{-a+s}\varphi^bVd\mu dt
                   \nonumber\\
                   &+Ca^{-\frac{(s-a)(p-1)}{s-p+1}}\int_0^\infty\int_M |\nabla\varphi|^{\frac{s-a}{s-p+1}p}\varphi^{b-p\frac{s-a}{s-p+1}}V^{-\frac{p-a-1}{s-p+1}} d\mu dt.
               \end{align}
               By a similar argument, we have
               \begin{align}
                   &b^q\left(\frac{a}{2}\right)^{-q+1}\int_0^\infty\int_Mu^{q-a-1}\varphi^{b-q}|\nabla\varphi|^qd\mu dt
                   \nonumber\\
                   \leq &\frac{1}{4}\int_0^\infty\int_Mu^{-a+s}\varphi^bVd\mu dt
                   \nonumber\\
                   &+Ca^{-\frac{(s-a)(q-1)}{s-q+1}}\int_0^\infty\int_M |\nabla\varphi|^{\frac{s-a}{s-q+1}q}\varphi^{b-q\frac{s-a}{s-q+1}}V^{-\frac{q-a-1}{s-q+1}} d\mu dt
               \end{align}
               and
               \begin{align}
                   &\frac{b}{1-a}\int_0^\infty\int_M u^{-a+1}\varphi^{b-1}|\partial_t \varphi |d\mu dt
                   \nonumber\\
                   \leq & \frac{1}{4}\int_0^\infty\int_Mu^{-a+s}\varphi^bVd\mu dt+C\int_0^\infty\int_M V^{-\frac{1-a}{s-1}}\varphi^{b-\frac{s-a}{s-1}}|\partial_t \varphi|^{\frac{s-a}{s-1}}d\mu dt.
               \end{align}
               From (2.7)-(2.10), we have
               \begin{align*}
                   &\frac{a}{2}\int_0^\infty \int_M |\nabla u|^pu^{-a-1}\varphi^bd\mu dt+\frac{a}{2}\int_0^\infty \int_M |\nabla u|^qu^{-a-1}\varphi^bd\mu dt
                   \nonumber\\
                   &\qquad +\frac{1}{4}\int_0^\infty\int_Mu^{-a+s}\varphi^bVd\mu dt
                   \nonumber\\
                   \leq & Ca^{-\frac{(s-a)(p-1)}{s-p+1}}\int_0^\infty\int_M |\nabla\varphi|^{\frac{s-a}{s-p+1}p}\varphi^{b-p\frac{s-a}{s-p+1}}V^{-\frac{p-a-1}{s-p+1}} d\mu dt
                   \nonumber\\
                   &+Ca^{-\frac{(s-a)(q-1)}{s-q+1}}\int_0^\infty\int_M |\nabla\varphi|^{\frac{s-a}{s-q+1}q}\varphi^{b-q\frac{s-a}{s-q+1}}V^{-\frac{q-a-1}{s-q+1}} d\mu dt
                   \nonumber\\
                   &+C\int_0^\infty\int_M V^{-\frac{1-a}{s-1}}\varphi^{b-\frac{s-a}{s-1}}|\partial_t \varphi|^{\frac{s-a}{s-1}}d\mu .
               \end{align*}
               Then (2.2) follows from $b>\mathrm{max}\{1,\frac{s}{s-1},\frac{ps}{s-p+1}\},0\leq \varphi\leq 1 $ and $a^a\leq C$.
           \end{proof}
           \begin{lemma}
               Let $0<a<\frac{1}{2}\mathrm{min}\{q-1,1,s-1,\frac{s-p+1}{p-1}\}$, $b>\mathrm{max}\{1,\frac{s+1}{s-1},\frac{2ps}{s-p+1}\} $ and $s>\mathrm{max}\{p-1,1\}$. Assume that $u$ is a nonnegative weak solution of (1.1). Then for any $ \varphi\in \mathrm{Lip}(S)$ with compact support and $0\leq \varphi\leq 1$, we have
               \begin{align}
                   &\int_0^\infty\int_MVu^s\varphi^bd\mu dt
                   \nonumber\\
                   \leq& C(a^{-1}Q)^{\frac{p-1}{p}}\cdot \left(\int\int_{S\setminus K}Vu^s\varphi^bd\mu dt\right)^{\frac{(a+1)(p-1)}{sp}}
                   \nonumber\\
                   &\cdot\left(\int\int_{S\setminus K}V^{-\frac{(a+1)(p-1)}{s-(p-1)(a+1)}}|\nabla\varphi|^{\frac{ps}{s-(a+1)(p-1)}} d\mu dt\right)^{\frac{s-(a+1)(p-1)}{sp}}
                   \nonumber\\
                    &+C(a^{-1}Q)^{\frac{q-1}{q}}\cdot \left(\int \int_{S\setminus K}Vu^s\varphi^bd\mu dt \right)^{\frac{(a+1)(q-1)}{sq}}
                   \nonumber\\
                   &\cdot\left(\int\int_{S\setminus K}V^{-\frac{(a+1)(q-1)}{s-(q-1)(a+1)}}|\nabla\varphi|^{\frac{qs}{s-(a+1)(q-1)}} d\mu dt\right)^{\frac{s-(a+1)(q-1)}{sq}}
                   \nonumber\\
                   &+C\left(\int \int_{S\setminus K}Vu^s\varphi^bd\mu dt \right)^{\frac{1}{s}}\cdot \left(\int\int_{S\setminus K}V^{-\frac{1}{s-1}}|\partial_t\varphi|^{\frac{s}{s-1}} d\mu dt\right)^{\frac{s-1}{s}},
               \end{align}
               where $K=\{x\in S:\varphi(x)=1\}$ and 
               \begin{align*}
                   Q=&a^{-\frac{(p-1)s}{s-p+1}}\int_0^\infty\int_M|\nabla \varphi|^{\frac{s-a}{s-p+1}p}V^{-\frac{p-a-1}{s-p+1}}d\mu dt
                   \nonumber\\
                   &+ a^{-\frac{(q-1)s}{s-q+1}}\int_0^\infty\int_M|\nabla \varphi|^{\frac{s-a}{s-q+1}q}V^{-\frac{q-a-1}{s-q+1}}d\mu dt
                   \nonumber\\
                   &+\int_0^\infty \int_M |\partial_t\varphi|^{\frac{s-a}{s-1}}V^{-\frac{-a+1}{s-1}}d\mu dt.
               \end{align*}               
           \end{lemma}
           \begin{proof}
               Without loss of generality, as in Lemma 2.2, we assume that $u$ is positive with $u^{-1}\in L^{\infty}_{loc}(M\times[0,\infty))$. Let $\psi=\varphi^b$, from (2.1), we have
               \begin{align}
                   &\int_0^\infty\int_MVu^s\varphi^b d\mu dt
                   \nonumber\\
                   \leq &b\int_0^\infty\int_M|\nabla u|^{p-2}\varphi^{b-1}\langle \nabla u,\nabla \varphi\rangle d\mu dt+b\int_0^\infty\int_M|\nabla u|^{q-2}\varphi^{b-1}\langle \nabla u,\nabla \varphi\rangle d\mu dt
                   \nonumber\\
                   &+b\int_0^\infty\int_Mu \varphi^{b-1}|\partial_t\varphi|d\mu dt-\int_Mu_0\varphi^b(x,0)d\mu.
               \end{align}
               Applying H$\ddot{\mathrm{o}}$lder inequality, we obtain
               \begin{align}
                   &\int_0^\infty\int_Mu\varphi^b|\partial_t \varphi|d\mu dt = \int_0^\infty \int_M u\varphi^{\frac{b}{s}}V^{\frac{1}{s}}\cdot V^{-\frac{1}{s}}\varphi^{\frac{s-1}{s}b-1}|\partial_t\varphi|d\mu dt
                   \nonumber\\
                   \leq & \left(\int\int_{S\setminus K} Vu^s\varphi^b d\mu dt \right)^{\frac{1}{s}}\cdot \left(\int\int_{S\setminus K}V^{-\frac{1}{s-1}}\varphi^{b-\frac{s}{s-1}}|\partial_t\varphi|^{\frac{s}{s-1}}d\mu dt \right)^{\frac{s-1}{s}}.
               \end{align}
               Using H$\ddot{\mathrm{o}}$lder inequality again with the pair ($\frac{p-1}{p},\frac{1}{p}$), we find that
               \begin{align}
                   &\int_0^\infty |\nabla u|^{p-1}\varphi^{b-1}|\nabla \varphi|d\mu dt
                   \nonumber\\
                   \leq & \left(\int_0^\infty \int_M\varphi^b|\nabla u|^pu^{-a-1}d\mu dt \right)^{\frac{p-1}{p}}\cdot \left(\int_0^\infty\int_M\varphi^{b-p}u^{(a+1)(p-1)}|\nabla \varphi|^pd\mu dt \right)^{\frac{1}{p}}.
               \end{align}
               Similarly, we have 
               \begin{align}
                   &\int_0^\infty |\nabla u|^{q-1}\varphi^{b-1}|\nabla \varphi|d\mu dt
                   \nonumber\\
                   \leq & \left(\int_0^\infty \int_M\varphi^b|\nabla u|^qu^{-a-1}d\mu dt \right)^{\frac{q-1}{q}}\cdot \left(\int_0^\infty\int_M\varphi^{b-q}u^{(a+1)(q-1)}|\nabla \varphi|^qd\mu dt \right)^{\frac{1}{q}}.
               \end{align}
               From Lemma 2.2 and the definition of $Q$, we have
                   \begin{align}
                   &\int_0^\infty \int_M |\nabla u|^pu^{-a-1}\varphi^bd\mu dt+\int_0^\infty \int_M |\nabla u|^qu^{-a-1}\varphi^bd\mu dt
                   \leq  Ca^{-1}Q.
               \end{align}
               Substituting (2.13)-(2.16) into (2.12), we obtain
               \begin{align}
                   &\int_0^\infty\int_MVu^s\varphi^b d\mu dt
                   \nonumber\\
                   \leq &C(a^{-1}Q)^{\frac{p-1}{p}}\cdot \left(\int_0^\infty\int_M\varphi^{b-p}u^{(a+1)(p-1)}|\nabla \varphi|^pd\mu dt \right)^{\frac{1}{p}}
                   \nonumber\\
                   &+ C(a^{-1}Q)^{\frac{q-1}{q}}\cdot \left(\int_0^\infty\int_M\varphi^{b-q}u^{(a+1)(q-1)}|\nabla \varphi|^qd\mu dt \right)^{\frac{1}{q}}
                   \nonumber\\
                   &+C\left(\int\int_{S\setminus K} Vu^s\varphi^bd\mu dt\right)^{\frac{1}{s}}\cdot \left(\int\int_{S\setminus K}V^{-\frac{1}{s-1}}\varphi^{b-\frac{s}{s-1}}|\partial_t\varphi|^{\frac{s}{s-1}}d\mu dt\right)^{\frac{s-1}{s}}.
               \end{align}
               Next we deal with the integrals
               \begin{align*}
                   \int_0^\infty\int_M\varphi^{b-p}u^{(a+1)(p-1)}|\nabla \varphi|^pd\mu dt,\quad \int_0^\infty\int_M\varphi^{b-q}u^{(a+1)(q-1)}|\nabla \varphi|^qd\mu dt.
               \end{align*}
               Through an application of H$\ddot{\mathrm{o}}$lder inequality, we derive
               \begin{align}
                   &\int_0^\infty\int_M\varphi^{b-p}u^{(a+1)(p-1)}|\nabla \varphi|^pd\mu dt
                   \nonumber\\
                   =&\int_0^\infty\int_M u^{(p-1)(a+1)}\varphi^{\frac{(p-1)(a+1)}{s}b}V^{\frac{(p-1)(a+1)}{s}}
                   \nonumber\\
                   &\qquad \cdot V^{-\frac{(p-1)(a+1)}{s}}\varphi^{\frac{s-(p-1)(a+1)}{s}b-p}|\nabla\varphi|^p d\mu dt
                   \nonumber\\
                   \leq &\left(\int\int_{S\setminus K} V^{-\frac{(p-1)(a+1)}{s-(p-1)(a+1)}}\varphi^{b-\frac{ps}{s-(p-1)(a+1)}}|\nabla \varphi|^{\frac{ps}{s-(p-1)(a+1)}}d\mu dt \right)^{\frac{s-(p-1)(a+1)}{s}}
                   \nonumber\\
                   &\qquad \cdot \left(\int\int_{S\setminus K} Vu^s\varphi^bd\mu dt\right)^{\frac{(p-1)(a+1)}{s}}.
               \end{align}
               In a similar way, we have
               \begin{align}
                   &\int_0^\infty\int_M\varphi^{b-q}u^{(a+1)(q-1)}|\nabla \varphi|^qd\mu dt
                   \nonumber\\
                   \leq &\left(\int\int_{S\setminus K} V^{-\frac{(q-1)(a+1)}{s-(q-1)(a+1)}}\varphi^{b-\frac{qs}{s-(q-1)(a+1)}}|\nabla \varphi|^{\frac{qs}{s-(q-1)(a+1)}}d\mu dt \right)^{\frac{s-(q-1)(a+1)}{s}}
                   \nonumber\\
                   &\qquad \cdot \left(\int\int_{S\setminus K} Vu^s\varphi^bd\mu dt\right)^{\frac{(q-1)(a+1)}{s}}.
               \end{align}
               Thus, from (2.17)-(2.19), we obtain
               \begin{align}
                   &\int_0^\infty\int_MVu^s\varphi^b d\mu dt
                   \nonumber\\
                   \leq &C(a^{-1}Q)^{\frac{p-1}{p}}\cdot \left(\int\int_{S\setminus K} Vu^s\varphi^bd\mu dt\right)^{\frac{(p-1)(a+1)}{sp}}
                   \nonumber\\
                   &\cdot \left(\int\int_{S\setminus K} V^{-\frac{(p-1)(a+1)}{s-(p-1)(a+1)}}\varphi^{b-\frac{ps}{s-(p-1)(a+1)}}|\nabla \varphi|^{\frac{ps}{s-(p-1)(a+1)}}d\mu dt \right)^{\frac{s-(p-1)(a+1)}{sp}}
                   \nonumber\\
                   &+ C(a^{-1}Q)^{\frac{q-1}{q}}\cdot \left(\int\int_{S\setminus K} Vu^s\varphi^bd\mu dt\right)^{\frac{(q-1)(a+1)}{sq}}
                   \nonumber\\
                   &\cdot \left(\int\int_{S\setminus K} V^{-\frac{(q-1)(a+1)}{s-(q-1)(a+1)}}\varphi^{b-\frac{qs}{s-(q-1)(a+1)}}|\nabla \varphi|^{\frac{qs}{s-(q-1)(a+1)}}d\mu dt \right)^{\frac{s-(q-1)(a+1)}{sq}}
                   \nonumber\\
                   &+C\left(\int\int_{S\setminus K} Vu^s\varphi^bd\mu dt\right)^{\frac{1}{s}}\cdot \left(\int\int_{S\setminus K}V^{-\frac{1}{s-1}}\varphi^{b-\frac{s}{s-1}}|\partial_t\varphi|^{\frac{s}{s-1}}d\mu dt\right)^{\frac{s-1}{s}}.
               \end{align}
               By the choice of $b$ and $0\leq \varphi\leq 1$, we complete the proof.
           \end{proof}

           \section{Proof of Theorem 1.5}
           Before the proof, we give a Lemma.
           \begin{lemma}
                Let $f$ be a nonnegative decreasing function on $\mathbb{R_+} $. 
                \\
                (1)  If (1.7) holds, then
                \begin{align}
                  &\int \int_{S\setminus B_R} f([r(x)^{\theta_2}+t^{\theta_1}]^{\frac{1}{\theta_2}})t^{(\theta_1-1)(\frac{s}{s-1}-\epsilon)}V^{-\bar{s}_{2}+\epsilon}d\mu dt
                  \nonumber\\
                  \leq & C\int_{R/2^{1/\theta_2}}^\infty f(r)r^{\bar{s}_{1}+C_0\epsilon-1}(\mathrm{log}\ r)^{s_{2}}dr
               \end{align}
                for $R$ large enough .
                \\
                (2) If (1.8) holds, then 
               \begin{align}
                  &\int \int_{S\setminus B_R} f([r(x)^{\theta_2}+t^{\theta_1}]^{\frac{1}{\theta_2}})r(x)^{(\theta_2-1)z(\frac{s}{s-z+1}-\epsilon)}V^{-\bar{s}_{4,z}+\epsilon}d\mu dt
                  \nonumber\\
                  \leq & C\int_{R/2^{1/\theta_2}}^\infty f(r)r^{\bar{s}_{3,z}+C_0\epsilon-1}(\mathrm{log}\ r)^{s_{4,z}}dr
               \end{align}
                for $z\in \{p,q\}$ and $R$ large enough .
                
           \end{lemma}
           \begin{proof}
            Since the proof is similar, we only prove (3.2). By the monotonicity of $f$, we have
               \begin{align*}
                   &\int \int_{S\setminus B_R} f([r(x)^{\theta_2}+t^{\theta_1}]^{\frac{1}{\theta_2}})r(x)^{(\theta_2-1)z(\frac{s}{s-z+1}-\epsilon)}V^{-\bar{s}_{4,z}+\epsilon}d\mu dt
                   \\
                   =&\sum\limits_{i=0}^\infty\int\int_{B_{2^{(i+1)/\theta_2}R}\setminus B_{2^{i/\theta_2}R}}f([r(x)^{\theta_2}+t^{\theta_1}]^{\frac{1}{\theta_2}})r(x)^{(\theta_2-1)z(\frac{s}{s-z+1}-\epsilon)}V^{-\bar{s}_{4,z}+\epsilon}d\mu dt
                   \\
                   \leq & \sum\limits_{i=0}^\infty f(2^{i/\theta_2}R)\int\int_{B_{2^{(i+1)/\theta_2}R}\setminus B_{2^{i/\theta_2}R}}r(x)^{(\theta_2-1)z(\frac{s}{s-z+1}-\epsilon)}V^{-\bar{s}_{4,z}+\epsilon}d\mu dt
                   \\
                   \leq & C\sum\limits_{i=0}^\infty f(2^{i/\theta_2}R)\cdot \left( 2^{i/\theta_2}R\right)^{\bar{s}_{3,z}+C_0\epsilon}\left(\mathrm{log}(2^{i/\theta_2}R)\right)^{s_{4,z}}
                   \\
                   \leq &C\sum\limits_{i=0}^\infty f(2^{i/\theta_2}R)\left( 2^{(i-1)/\theta_2}R\right)^{\bar{s}_{3,z}+C_0\epsilon-1}\left(\mathrm{log}(2^{(i-1)/\theta_2}R)\right)^{s_{4,z}}(2^{i/\theta_2}-2^{(i-1)/\theta_2})R
                   \\
                   \leq &C\sum\limits_{i=-1}^\infty\int_{2^{i/\theta_2}R}^{2^{(i+1)/\theta_2}R}f(r)r^{\bar{s}_{3,z}+C_0\epsilon-1}(\mathrm{log}\ r)^{s_{4,z}}dr
                   \\
                   =&C\int_{R/2^{1/\theta_2}}^\infty f(r)r^{\bar{s}_{3,z}+C_0\epsilon-1}(\mathrm{log}\ r)^{s_{4,z}}dr.
               \end{align*}
           \end{proof}
           Now we begin the proof of Theorem 1.5.
           
           ~\\
         $\mathit{Proof\ of\ Theorem\ 1.5.}$ For $R>0$ large enough, let $a=\frac{1}{\mathrm{log R}}$. Define $\varphi_n=\varphi(x)\eta_n(x)$ for $n\in \mathbb{N}$, where
         \begin{align*}
    \varphi(x) &= 
    \begin{cases}
        1, & \text{if }(x,t)\in E_R, \\
        \left(\frac{r(x)^{\theta_2}+t^{\theta_1}}{R^{\theta_2}}\right)^{-C_1a},  & \text{if } (x,t)\in S\setminus E_R,
    \end{cases} \\
\end{align*}
    and 
         \begin{align*}
    \eta_n(x) &= 
    \begin{cases}
        1, & \text{if } (x,t)\in E_{nR}, \\
        2-\frac{r(x)^{\theta_2}+t^{\theta_1}}{(nR)^{\theta_2}},  & \text{if } (x,t)\in E_{2^{1/\theta_2}nR}\setminus E_{nR},
        \\
        0,  & \text{if }(x,t)\in S \setminus E_{2^{1/\theta_2}nR}.
    \end{cases} \\
\end{align*}
Here we assume that $C_1>\frac{C_0+\theta_2+2}{\theta_2}$ with $C_0$ as in HP1.
Note that $\varphi_n\xrightarrow{} \varphi$ as $n\xrightarrow{}\infty$ and
\begin{align}
    |\nabla \varphi_n|^\alpha\leq C(|\nabla \varphi|^\alpha+\varphi^\alpha|\nabla\eta_n|^\alpha) ,\quad |\partial_t\varphi_n|^\alpha\leq C(|\partial_t \varphi|^\alpha+\varphi^\alpha|\partial_t\eta_n|^\alpha)\quad 
\end{align}
a.e. on $S$ for $\alpha\geq 1$. Choosing $\varphi_n$ as the test function in (2.2) and letting $b>\mathrm{max}\{1,\frac{s+1}{s-1},\frac{2ps}{s-p+1}\}$, we have
\begin{align}
                   &\int_0^\infty\int_Mu^{-a+s}\varphi_n^b Vd\mu dt
                   \nonumber\\
                   \leq & Ca^{-\frac{(p-1)s}{s-p+1}}\int_0^\infty\int_M|\nabla \varphi_n|^{\frac{s-a}{s-p+1}p}V^{-\frac{p-a-1}{s-p+1}}d\mu dt
                   \nonumber\\
                   &+ Ca^{-\frac{(q-1)s}{s-q+1}}\int_0^\infty\int_M|\nabla \varphi_n|^{\frac{s-a}{s-q+1}q}V^{-\frac{q-a-1}{s-q+1}}d\mu dt
                   \nonumber\\
                   &+C\int_0^\infty \int_M |\partial_t\varphi_n|^{\frac{s-a}{s-1}}V^{-\frac{-a+1}{s-1}}d\mu dt
                   \nonumber\\
                   \leq &C(I_{p,1}+I_{p,2})+C(I_{q,1}+I_{q,2})+(I_{t,1}+I_{t,2}),
               \end{align}
               where
               \begin{align}
                   &I_{p,1}=a^{-\frac{(p-1)s}{s-p+1}}\int\int_{S\setminus E_R}|\nabla \varphi|^{\frac{s-a}{s-p+1}p}V^{-\frac{p-a-1}{s-p+1}}d\mu dt,
                   \\
                   &I_{p,2}=a^{-\frac{(p-1)s}{s-p+1}}\int\int_{E_{2^{1/\theta_2}nR}\setminus E_{nR}}\varphi^{\frac{s-a}{s-p+1}p}|\nabla \eta_n|^{\frac{s-a}{s-p+1}p}V^{-\frac{p-a-1}{s-p+1}}d\mu dt,
                   \\
                   &I_{q,1}=a^{-\frac{(q-1)s}{s-q+1}}\int\int_{S\setminus E_R}|\nabla \varphi|^{\frac{s-a}{s-q+1}q}V^{-\frac{q-a-1}{s-q+1}}d\mu dt,
                   \\
                   &I_{q,2}=a^{-\frac{(q-1)s}{s-q+1}}\int\int_{E_{2^{1/\theta_2}nR}\setminus E_{nR}}\varphi^{\frac{s-a}{s-q+1}q}|\nabla \eta_n|^{\frac{s-a}{s-q+1}q}V^{-\frac{q-a-1}{s-q+1}}d\mu dt,
                   \\
                   &I_{t,1}=\int\int_{S\setminus E_R}|\partial_t\varphi|^{\frac{s-a}{s-1}}V^{-\frac{-a+1}{s-1}}d\mu dt
                   \\
                   &I_{t,2}=\int\int_{E_{2^{1/\theta_2}nR}\setminus E_{nR}} \varphi^{\frac{s-a}{s-1}}|\partial_t\eta_n|^{\frac{s-a}{s-1}}V^{-\frac{-a+1}{s-1}}d\mu dt.
               \end{align}
               By Lemma 3.1 and 
               \begin{align*}
                   |\nabla \varphi|\leq &C_1a\theta_2\left(\frac{r(x)^{\theta_2}+t^{\theta_1}}{R^{\theta_2}}\right)^{-C_1a-1}\frac{r(x)^{\theta_2-1}}{R^{\theta_2}}
                   \\
                   \leq &Ca(r(x)^{\theta_2}+t^{\theta_1})^{-C_1a-1}r(x)^{\theta_2-1}
               \end{align*}
               a.e. on $S$, we have
               \begin{align*}
                   I_{p,1}\leq & C a^{-\frac{(p-1)s}{s-p+1}}\int\int_{S\setminus E_R}\left(a(r(x)^{\theta_2}+t^{\theta_1})^{-C_1a-1}r(x)^{\theta_2-1}\right)^{\frac{s-a}{s-p+1}p}V^{-\frac{p-a-1}{s-p+1}}d\mu dt
                   \\
                   \leq &C a^{\frac{p(s-a)-(p-1)s}{s-p+1}}\int\int_{S\setminus E_R}(r(x)^{\theta_2}+t^{\theta_1})^{-(C_1a+1)\frac{p(s-a)}{s-p+1}}
                   \\
                   & \qquad \cdot r(x)^{(\theta_2-1)\frac{p(s-a)}{s-p+1}}V^{-\frac{p-a-1}{s-p+1}}d\mu dt
                   \\
                   \leq &Ca^{\frac{p(s-a)-(p-1)s}{s-p+1}}\int_{R/2^{1/\theta_2}}^\infty r^{-(C_1a+1)\theta_2\frac{p(s-a)}{s-p+1}+\bar{s}_{3,p}+C_0\frac{a}{s-p+1}-1}(\mathrm{log}\ r)^{s_{4,p}}dr.
               \end{align*}
               Let $y=\alpha\mathrm{log}\ r$, where
               \begin{align*}
                   -\alpha=&-(C_1a+1)\theta_2\frac{p(s-a)}{s-p+1}+\bar{s}_{3,p}+C_0\frac{a}{s-p+1}.
               \end{align*}
               By the choice of $C_1, \bar{s}_{3,p}$ and $a$ small enough, we have
               \begin{align*}
                   -\alpha \leq  -C_1a\theta_2\frac{p(s-a)}{s-p+1}+\frac{p\theta_2a}{s-p+1}+C_0\frac{a}{s-p+1}\leq -\frac{a}{s-p+1}.
               \end{align*}
                Thus we get
               \begin{align}
                   I_{p,1}\leq & Ca^{\frac{p(s-a)-(p-1)s}{s-p+1}}\int_{0}^\infty e^{-y}\left(\frac{y}{\alpha}\right)^{s_{4,p}}\frac{dy}{\alpha}
                   \leq Ca^{\frac{p-1}{s-p+1}-s_{4,p}}.
               \end{align}
               Under the condition (1.8) with $\epsilon=\frac{a}{s-p+1}$ and $|\nabla \eta_n|\leq \frac{\theta_2r(x)^{\theta_2-1}}{(nR)^{\theta_2}}$,
                we have
               \begin{align*}
                   I_{p,2} \leq &C a^{-\frac{(p-1)s}{s-p+1}}\int\int_{E_{2^{1/\theta_2}nR}\setminus E_{nR}}n^{-C_1\theta_2a\frac{s-a}{s-p+1}p}\left(\frac{r(x)^{\theta_2-1}}{(nR)^{\theta_2}}\right)^{\frac{s-a}{s-p+1}p}V^{-\frac{p-a-1}{s-p+1}}d\mu dt
                   \\
                   \leq& Ca^{-\frac{(p-1)s}{s-p+1}}n^{-C_1\theta_2a\frac{s-a}{s-p+1}p}(nR)^{ -\theta_2\frac{s-a}{s-p+1}p}
                   \\
                   &\qquad \cdot\int\int_{E_{2^{1/\theta_2}nR}\setminus E_{nR}}r(x)^{ (\theta_2-1)\frac{s-a}{s-p+1}p}V^{-\frac{p-a-1}{s-p+1}}d\mu dt
                   \\
                   \leq& Ca^{-\frac{(p-1)s}{s-p+1}}n^{-C_1\theta_2a\frac{s-a}{s-p+1}p}(nR)^{ -\theta_2\frac{s-a}{s-p+1}p} (nR)^{\bar{s}_{3,p}+C_0\frac{a}{s-p+1}}(\mathrm{log}\ R)^{s_{4,p}}
                   \\
                   \leq & Ca^{-\frac{(p-1)s}{s-p+1}}n^{-C_1\theta_2a\frac{s-a}{s-p+1}p +\theta_2\frac{a}{s-p+1}p+C_0\frac{a}{s-p+1}}(\mathrm{log}\ R)^{s_{4,p}},
               \end{align*}
               where the last inequality follows from $\bar{s}_{3,p}=\frac{p\theta_2s}{s-p+1}$ and $R^a=e$. From the choice of $C_1$ and $a$ small enough, we get
               \begin{align*}
                   {-C_1a\frac{s-a}{s-p+1}p +\theta_2\frac{a}{s-p+1}p+C_0\frac{a}{s-p+1}}\leq -\frac{a}{s-p+1}.
               \end{align*}
               Therefore we have
               \begin{align*}
                   I_{p,2}\leq Ca^{-\frac{(p-1)s}{s-p+1}} n^{-\frac{a}{s-p+1} }(\mathrm{log}\ nR)^{s_{4,p}}.
               \end{align*}
               Similarly, for $z=q$, we obtain 
               \begin{align*}
                   I_{q,1}\leq Ca^{\frac{q-1}{s-q+1}-s_{4,q}},\quad I_{q,2}\leq C a^{-\frac{(q-1)s}{s-q+1}}n^{-\frac{a}{s-q+1} }(\mathrm{log}\ nR)^{s_{4,q}}.
               \end{align*}
               Next we estimate $I_{t,1}, I_{t,2}$. By (3.1) and 
               \begin{align*}
                   |\partial_t\varphi|\leq &C_1a\theta_2\left(\frac{r(x)^{\theta_2}+t^{\theta_1}}{R^{\theta_2}}\right)^{-C_1a-1}\frac{t^{\theta_1-1}}{R^{\theta_2}}
                   \\
                   \leq &Ca(r(x)^{\theta_2}+t^{\theta_1})^{-C_1a-1}t^{\theta_1-1},
               \end{align*}
               we find that
               \begin{align*}
                   I_{t,1}=&\int\int_{S\setminus E_R}|\partial_t\varphi|^{\frac{s-a}{s-1}}V^{-\frac{-a+1}{s-1}}d\mu dt                   
                   \\
                   \leq & C \int\int_{S\setminus E_R}\left(a(r(x)^{\theta_2}+t^{\theta_1})^{-C_1a-1}t^{\theta_1-1}\right)^{\frac{s-a}{s-1}}V^{-\frac{-a+1}{s-1}}d\mu d t
                   \\
                   \leq & Ca ^{\frac{s-a}{s-1}}\int\int_{S\setminus E_R}\left([r(x)^{\theta_2}+t^{\theta_1}]^{\frac{1}{\theta_2}}\right)^{-\theta_2(C_1a+1)\frac{s-a}{s-1}}t^{(\theta_1-1)\frac{s-a}{s-1}}V^{-\frac{-a+1}{s-1}}d\mu d t
                   \\
                   \leq & Ca ^{\frac{s-a}{s-1}}\int_{R/2^{1/\theta_2}}^\infty r^{-\theta_2(C_1a+1)\frac{s-a}{s-1}+\bar{s}_1+C_0\frac{a}{s-1}-1 }(\mathrm{log}\ r)^{s_2}dr.
               \end{align*}
               Let $y=\alpha \mathrm{log}\ r$, where
               \begin{align*}
                   -\alpha=&-\theta_2(C_1a+1)\frac{s-a}{s-1}+\bar{s}_1+C_0\frac{a}{s-1}.
               \end{align*}
               From the choice of $C_1,\bar{s}_1$ and $a$ small enough, we obtain
               \begin{align*}
                   -\alpha\leq -\frac{a}{s-1}.
               \end{align*}
               Hence 
               \begin{align*}
                   I_{t,1}\leq Ca ^{\frac{s-a}{s-1}}\int_{R/2^{1/\theta_2}}^\infty e^{-y}\left( \frac{y}{\alpha}\right)^{\bar{s}_2}\frac{dy}{\alpha}\leq Ca^{\frac{1}{s-1}-s_2}.
               \end{align*}
               Since
               \begin{align*}
                   |\partial_t \eta_n|\leq \frac{\theta_1t^{\theta_1-1}}{(nR)^{\theta_2}},
               \end{align*}
               using (1.7) with $\epsilon=\frac{a}{s-1}$, we deduce
               \begin{align*}
                   I_{t,2}\leq & C\int\int_{E_{2^{1/\theta_2}nR}\setminus E_{nR}} (n^{-\theta_2C_1a})^{\frac{s-a}{s-1}}\left(\frac{\theta_1t^{\theta_1-1}}{(nR)^{\theta_2}}\right)^{\frac{s-a}{s-1}}V^{-\frac{-a+1}{s-1}}d\mu dt
                   \\
                   \leq & C (nR)^{-\theta_2\frac{s-a}{s-1}}n^{-C_1\theta_2a\frac{s-a}{s-1}}(nR)^{\bar{s}_1+C_0\epsilon}(\mathrm{log}(nR))^{s_2}
                   \\
                   \leq & Cn^{-\frac{a}{s-1}(-\theta_2+C_1\theta_2s-C_1\theta_2a-C_0)}(\mathrm{log}(nR))^{s_2}
                   \\
                   \leq & Cn^{-\frac{a}{s-1}}(\mathrm{log}(nR))^{s_2}.
               \end{align*}
             Here in the last inequality, we use the fact $C_1>\frac{C_0+\theta_2+2}{\theta_2}$ and $a$ small enough. By the estimates of $I_{p,1},I_{p,2},I_{q,1},I_{q,2},I_{t,1},I_{t,2}$, we have
             \begin{align*}
               & \int_0^\infty\int_Mu^{-a+s}\varphi_n^b Vd\mu dt 
               \\
               \leq & C(a^{\frac{p-1}{s-p+1}-s_{4,p}}+ a^{-\frac{(p-1)s}{s-p+1}} n^{-\frac{a}{s-p+1} }(\mathrm{log}\ (nR))^{s_{4,p}})
               \\
               &+C(a^{\frac{q-1}{s-q+1}-s_{4,q}}+ a^{-\frac{(q-1)s}{s-q+1}} n^{-\frac{a}{s-q+1} }(\mathrm{log}\ (nR))^{s_{4,q}})
               \\
               &+C(a^{\frac{1}{s-1}-s_2}+ n^{-\frac{a}{s-1}}(\mathrm{log}(nR))^{s_2}).
             \end{align*}
             Letting $n\xrightarrow{}\infty$, we obtain
             \begin{align*}
                 \int\int_{E_R}u^{-a+s} Vd\mu dt\leq C(a^{\frac{p-1}{s-p+1}-s_{4,p}}+a^{\frac{q-1}{s-q+1}-s_{4,q}}+a^{\frac{1}{s-1}-s_2}).
             \end{align*}
             By the choice of $s_{4,p},s_{4,q}$ and $s_2$ and letting $R\xrightarrow{}\infty$, we conclude that
             \begin{align*}
                 \int_0^\infty\int_Mu^{s} Vd\mu dt=0,
             \end{align*}
             which implies that $u=0$ a.e on $S$.   {\qed}

             \section{Proof of Theorem 1.6}
              Analogous to Lemma 3.1, we give the following lemma. Since the proof is similar, we omit it.
             \begin{lemma}
                Let $f$ be a nonnegative decreasing function on $\mathbb{R_+} $. 
                \\
                (1)  If (1.9) holds, then
                \begin{align}
                  &\int \int_{S\setminus B_R} f([r(x)^{\theta_2}+t^{\theta_1}]^{\frac{1}{\theta_2}})t^{(\theta_1-1)(\frac{s}{s-1}-\epsilon)}V^{-\bar{s}_{2}+\epsilon}d\mu dt
                  \nonumber\\
                  \leq & C\int_{R/2^{1/\theta_2}}^\infty f(r)r^{\bar{s}_{1}+C_0\epsilon-1}(\mathrm{log}\ r)^{\bar{s}_{2}}dr
               \end{align}
                for $R$ large enough. In particular, from Remark 1.2, (4.1) holds for $\epsilon=0$.
                \\
                (2) If (1.10) holds, then 
               \begin{align}
                  &\int \int_{S\setminus B_R} f([r(x)^{\theta_2}+t^{\theta_1}]^{\frac{1}{\theta_2}})r(x)^{(\theta_2-1)z(\frac{s}{s-z+1}-\epsilon)}V^{-\bar{s}_{4,z}+\epsilon}d\mu dt
                  \nonumber\\
                  \leq & C\int_{R/2^{1/\theta_2}}^\infty f(r)r^{\bar{s}_{3,z}+C_0\epsilon-1}(\mathrm{log}\ r)^{\bar{s}_{4,z}}dr
               \end{align}
                for $z\in \{p,q\}$ and $R$ large enough.
                \\
                (3) If (1.11) holds, then
               \begin{align}
                  &\int \int_{S\setminus B_R} f([r(x)^{\theta_2}+t^{\theta_1}]^{\frac{1}{\theta_2}})r(x)^{(\theta_2-1)z(\frac{s}{s-z+1}+\epsilon)}V^{-\bar{s}_{4,z}-\epsilon}d\mu dt
                  \nonumber\\
                  \leq & C\int_{R/2^{1/\theta_2}}^\infty f(r)r^{\bar{s}_{3,z}+C_0\epsilon-1}(\mathrm{log}\ r)^{\bar{s}_{4,z}}dr
               \end{align}
                for $z\in \{p,q\}$ and $R$ large enough.
           \end{lemma}

           ~\\
         $\mathit{Proof\ of\ Theorem\ 1.6.}$ For $R>0$ large enough, let $a=\frac{1}{\mathrm{log R}}, b>\mathrm{max}\{1,\frac{s-1}{s+1},\frac{2sp}{s-p+1}\}$ and $C_1>\mathrm{max} \{{\frac{C_0+\theta_2+1}{\theta_2}},\frac{2(C_0+1)}{\theta_2(s-p+1)},\frac{2(C_0+1)}{\theta_2(s-1)s}\}$. From Lemma 2.3, we have
         \begin{align}
                   &\int_0^\infty\int_MVu^s\varphi_n^bd\mu dt
                   \nonumber\\
                   \leq& C(a^{-1}Q)^{\frac{p-1}{p}}\cdot \left(\int\int_{S\setminus K}Vu^s\varphi_n^bd\mu dt\right)^{\frac{(a+1)(p-1)}{sp}}
                   \nonumber\\
                   &\cdot\left(\int\int_{S\setminus K}V^{-\frac{(a+1)(p-1)}{s-(p-1)(a+1)}}|\nabla\varphi_n|^{\frac{ps}{s-(a+1)(p-1)}} d\mu dt\right)^{\frac{s-(a+1)(p-1)}{sp}}
                   \nonumber\\
                    &+C(a^{-1}Q)^{\frac{q-1}{q}}\cdot \left(\int \int_{S\setminus K}Vu^s\varphi_n^bd\mu dt \right)^{\frac{(a+1)(q-1)}{sq}}
                   \nonumber\\
                   &\cdot\left(\int\int_{S\setminus K}V^{-\frac{(a+1)(q-1)}{s-(q-1)(a+1)}}|\nabla\varphi_n|^{\frac{qs}{s-(a+1)(q-1)}} d\mu dt\right)^{\frac{s-(a+1)(q-1)}{sq}}
                   \nonumber\\
                   &+C\left(\int \int_{S\setminus K}Vu^s\varphi_n^bd\mu dt \right)^{\frac{1}{s}}\cdot \left(\int\int_{S\setminus K}V^{-\frac{1}{s-1}}|\partial_t\varphi_n|^{\frac{s}{s-1}} d\mu dt\right)^{\frac{s-1}{s}},
               \end{align}
               where $\varphi_n$ are the functions defined in Section 3 and 
               \begin{align*}
                   Q=&a^{-\frac{(p-1)s}{s-p+1}}\int_0^\infty\int_M|\nabla \varphi_n|^{\frac{s-a}{s-p+1}p}V^{-\frac{p-a-1}{s-p+1}}d\mu dt
                   \nonumber\\
                   &+ a^{-\frac{(q-1)s}{s-q+1}}\int_0^\infty\int_M|\nabla \varphi_n|^{\frac{s-a}{s-q+1}q}V^{-\frac{q-a-1}{s-q+1}}d\mu dt
                   \nonumber\\
                   &+\int_0^\infty \int_M |\partial_t\varphi_n|^{\frac{s-a}{s-1}}V^{-\frac{-a+1}{s-1}}d\mu dt
                   \\
                   =&J_{p,1}+J_{q,1}+J_{t,1}.
               \end{align*} 
               We want to estimate the term in (4.4) except $ \int \int_{S\setminus K}Vu^s\varphi_n^bd\mu dt$. Firstly, we estimate $ J_{p,1},J_{q,1},J_{t,1}$. As shown in Section 3, we have
               \begin{align*}
                   J_{p,1}\leq C(I_{p,1}+I_{p,2}),
               \end{align*}
               where $I_{p,1} $ and $ I_{p,2} $ are given by (3.5) and (3.6). Through the same argument, the only difference is that we use (4.2) rather than (3.2), we have 
               \begin{align}
                   J_{p,1}\leq C( a^{\frac{p-1}{s-p+1}-\bar{s}_{4,p}}+a^{-\frac{(p-1)s}{s-p+1}} n^{-\frac{a}{s-p+1} }(\mathrm{log}\ (nR))^{\bar{s}_{4,p}}).
               \end{align}
               Similarly, using Lemma 4.1, we have
               \begin{align}
                   &J_{q,1}\leq C(I_{q,1}+I_{q,2})\leq C( a^{\frac{q-1}{s-q+1}-\bar{s}_{4,q}}+a^{-\frac{(q-1)s}{s-q+1}} n^{-\frac{a}{s-q+1} }(\mathrm{log}\ (nR))^{\bar{s}_{4,q}}),
                   \\
                   & J_{t,1}\leq C(I_{t,1}+I_{t,2})\leq C( a^{\frac{1}{s-1}-\bar{s}_2}+n^{-\frac{a}{s-1}}(\mathrm{log}(nR))^{\bar{s}_2}).
               \end{align}
               From (4.5)-(4.7) and the definition of $\bar{s}_{4,p},\bar{s}_{4,q}, \bar{s}_2$, we have
               \begin{align*}
                   Q\leq& C(1+a^{-\frac{(p-1)s}{s-p+1}} n^{-\frac{a}{s-p+1} }(\mathrm{log}\ (nR))^{\bar{s}_{4,p}}+ a^{-\frac{(q-1)s}{s-q+1}} n^{-\frac{a}{s-q+1} }(\mathrm{log}\ (nR))^{\bar{s}_{4,q}}
                   \\
                   &\qquad+n^{-\frac{a}{s-1}}(\mathrm{log}(nR))^{\bar{s}_2}).
               \end{align*}
               Secondly, we give the estimate of the following integral
               \begin{align*}
                   &J_{p,2}=\int\int_{S\setminus K}V^{-\frac{(a+1)(p-1)}{s-(p-1)(a+1)}}|\nabla\varphi_n|^{\frac{ps}{s-(a+1)(p-1)}} d\mu dt
                   \\
                   \leq & C\int\int_{S\setminus E_R}V^{-\frac{(a+1)(p-1)}{s-(p-1)(a+1)}}|\nabla\varphi|^{\frac{ps}{s-(a+1)(p-1)}} d\mu dt
                   \\
                   &+C\int\int_{S\setminus E_R}V^{-\frac{(a+1)(p-1)}{s-(p-1)(a+1)}}|\nabla\eta_n|^{\frac{ps}{s-(a+1)(p-1)}} \varphi^{ \frac{ps}{s-(a+1)(p-1)}}d\mu dt
                   \\
                   =&C(J_{p,3}+J_{p,4}).
               \end{align*}
               From $ |\nabla \varphi|\leq Ca(r(x)^{\theta_2}+t^{\theta_1})^{-C_1a-1}r(x)^{\theta_2-1}$, we have
               \begin{align*}
                   J_{p,3}\leq &  C \int\int_{S\setminus E_R}V^{-\frac{(a+1)(p-1)}{s-(p-1)(a+1)}}(a(r(x)^{\theta_2}+t^{\theta_1})^{-C_1a-1}r(x)^{\theta_2-1})^{\frac{ps}{s-(a+1)(p-1)}} d\mu dt
               \end{align*}
               Let $\delta_p=\frac{(a+1)(p-1)}{s-(p-1)(a+1)}-\bar{s}_{4,p}$, then
               \begin{align*}
                   \frac{(p-1)sa}{(s-p+1)^2}\leq \delta_p=\frac{(p-1)sa}{(s-(a+1)(p-1))(s-p+1)}\leq \frac{2(p-1)sa}{(s-p+1)^2}
               \end{align*}
               for $a$ small enough and 
               \begin{align*}
                   \frac{ps}{s-(a+1)(p-1)}=\frac{\bar{s}_{3,p}}{\theta_2}+p\delta_p.
               \end{align*}
               Using (4.3) with $\epsilon=\delta_p$, we have
               \begin{align*}
                   J_{p,3}\leq &C a^{\frac{\bar{s}_{3,p}}{\theta_2}+p\delta_p} \int\int_{S\setminus E_R}(r(x)^{\theta_2}+t^{\theta_1})^{-(C_1a+1)(\frac{\bar{s}_{3,p}}{\theta_2}+p\delta_p)}
                   \\
                   &\qquad \cdot r(x)^{(\theta_2-1)p(\frac{s}{s-p+1}+\delta_p)}V^{-\frac{p-1}{s-p+1}-\delta_p} d\mu dt
                   \\
                   \leq & C a^{\frac{\bar{s}_{3,p}}{\theta_2}+p\delta_p}\int_{R/2^{1/\theta_2}}^\infty r^{-\theta_2(C_1a+1)(\frac{\bar{s}_{3,p}}{\theta_2}+p\delta_p)+\bar{s}_{3,p}+C_0\delta_p-1}(\mathrm{log}\ r)^{\bar{s}_{4,p}}dr.
               \end{align*}
               Define $y=\alpha \mathrm{log}\ r$, where
               \begin{align*}
                   -\alpha=&-\theta_2(C_1a+1)(\frac{\bar{s}_{3,p}}{\theta_2}+p\delta_p)+\bar{s}_{3,p}+C_0\delta_p\\
                   \leq& -\frac{psa}{(s-p+1)^2},
                   \end{align*}
               since $C_1>\mathrm{max} \{{\frac{C_0+\theta_2+1}{\theta_2}},\frac{2(C_0+1)}{\theta_2(s-p+1)},\frac{2(C_0+1)}{\theta_2(s-1)s}\}$ and $a$ small enough. Hence, we have
               \begin{align*}
                   J_{p,3}\leq &C a^{\frac{\bar{s}_{3,p}}{\theta_2}+p\delta_p}\int_0^\infty e^{-y}\left(\frac{y}{\alpha}\right)^{\bar{s}_{4,p}}\frac{dy}{\alpha}
                   \\
                   \leq & Ca^{\frac{ps}{s-(a+1)(p-1)}-\bar{s}_{4,p}-1 }
                   \\
                   =&C a^{\frac{(p-1)s}{s-(a+1)(p-1)}+\frac{(p-1)sa}{(s-p+1)(s-(p-1)(a+1))} }
                   \\
                   \leq  &C a^{\frac{(p-1)s}{s-(a+1)(p-1)}}
               \end{align*}
               for $a$ small enough. Since  $|\nabla \eta_n|\leq \frac{\theta_2r(x)^{\theta_2-1}}{(nR)^{\theta_2}}$ and using (1.11), we obtain
               \begin{align*}
                   J_{p,4}\leq &\int\int_{E_{2^{1/\theta_2}nR} \setminus E_R}V^{-\frac{(a+1)(p-1)}{s-(p-1)(a+1)}}\left(\frac{\theta_2r(x)^{\theta_2-1}}{(nR)^{\theta_2}}\right)^{\frac{ps}{s-(a+1)(p-1)}} \varphi^{ \frac{ps}{s-(a+1)(p-1)}}d\mu dt
                   \\
                   \leq & C n^{ -\frac{C_1a\theta_2ps}{s-(p-1)(a+1)}}(nR)^{-\frac{ps\theta_2}{s-(p-1)(a+1)}}
                   \\
                   &\qquad \cdot\int\int_{E_{2^{1/\theta_2}nR} \setminus E_R}V^{-\frac{(a+1)(p-1)}{s-(p-1)(a+1)}}r(x)^{(\theta_2-1)p(\frac{s}{s-p+1}+\delta_p)}d\mu dt
                   \\
                   \leq & C n^{ -\frac{C_1a\theta_2ps}{s-(p-1)(a+1)}}(nR)^{-\frac{ps\theta_2}{s-(p-1)(a+1)}} (nR)^{\frac{\theta_2ps}{s-p+1}+C_0\delta_p}(\mathrm{log}\ (nR))^{\bar{s}_{4,p}}.
               \end{align*}
               By the definition of $C_1$, we have
                \begin{align*}
                    &-\frac{C_1a\theta_2ps}{s-(p-1)(a+1)}-\frac{ps\theta_2}{s-(p-1)(a+1)}+\frac{\theta_2ps}{s-p+1}+C_0\delta_p
                    \\
                    \leq & -\frac{pas}{(s-p+1)(s-(p-1)(a+1))}\leq -\frac{pas}{(s-p+1)^2}
                \end{align*}
                for $a$ small enough. Note that $R^a=e$, we find that
                \begin{align*}
                    J_{p,4}\leq C n^{-\frac{pas}{(s-p+1)^2} }(\mathrm{log}\ (nR))^{\bar{s}_{4,p}}.
                \end{align*}
                Therefore, we derive
                \begin{align*}
                    J_{p,2}\leq C(J_{p,3}+J_{p,4})\leq C(a^{\frac{(p-1)s}{s-(a+1)(p-1)}}+n^{-\frac{pas}{(s-p+1)^2} }(\mathrm{log}\ (nR))^{\bar{s}_{4,p}} ).
                \end{align*}
                Similarly, we have
                \begin{align*}
                    J_{q,2}\leq  C(a^{\frac{(q-1)s}{s-(a+1)(q-1)}}+n^{-\frac{qas}{(s-q+1)^2} }(\mathrm{log}\ (nR))^{\bar{s}_{4,q}} ).
                \end{align*}
                Note that (1.9) holds for $\epsilon=0$, by the same argument in estimating $J_{t,1}$, we deduce
                \begin{align*}
                    J_{t,2}=&\int\int_{S\setminus K}V^{-\frac{1}{s-1}}|\partial_t\varphi_n|^{\frac{s}{s-1}} d\mu dt\leq C( a^{\frac{1}{s-1}-\bar{s}_2}+n^{-\frac{C_1\theta_2a}{s-1}}(\mathrm{log}(nR))^{\bar{s}_2})\\
                    \leq &C(1+n^{-\frac{C_1\theta_2a}{s-1}}(\mathrm{log}(nR))^{\bar{s}_2}).
                \end{align*}
                Return to (4.4), we have
                \begin{align}
                    &\int_0^\infty\int_MVu^s\varphi_n^bd\mu dt
                   \nonumber\\
                   \leq& C(a^{-1}Q)^{\frac{p-1}{p}}\cdot \left(\int\int_{S\setminus K}Vu^s\varphi_n^bd\mu dt\right)^{\frac{(a+1)(p-1)}{sp}}\cdot J_{p,2} ^{\frac{s-(a+1)(p-1)}{sp}}
                   \nonumber\\
                    &+C(a^{-1}Q)^{\frac{q-1}{q}}\cdot \left(\int \int_{S\setminus K}Vu^s\varphi_n^bd\mu dt \right)^{\frac{(a+1)(q-1)}{sq}}\cdot J_{q,2}^{\frac{s-(a+1)(q-1)}{sq}}
                   \nonumber\\
                   &+C\left(\int \int_{S\setminus K}Vu^s\varphi_n^bd\mu dt \right)^{\frac{1}{s}}\cdot J_{t,2}^{\frac{s-1}{s}}.
                \end{align}
                This yields that there exists a constant $\mathrm{max}\{\frac{(a+1)(p-1)}{sp},\frac{1}{s}\}<\gamma<1$ such that
                \begin{align}
                    &1+\int_0^\infty\int_MVu^s\varphi_n^bd\mu dt
                    \nonumber\\
                    \leq & C(1+(a^{-1}Q)^{\frac{p-1}{p}} J_{p,2} ^{\frac{s-(a+1)(p-1)}{sp}}+(a^{-1}Q)^{\frac{q-1}{q}}J_{q,2}^{\frac{s-(a+1)(q-1)}{sq}}+J_{t,2}^{\frac{s-1}{s}})
                    \nonumber\\
                    & \qquad \cdot \left(1+\int \int_{S\setminus K}Vu^s\varphi_n^bd\mu dt \right)^\gamma,
                \end{align}
                that is,
                \begin{align}
                    &\left(1+\int_0^\infty\int_MVu^s\varphi_n^bd\mu dt\right)^{1-\gamma}
                    \nonumber\\
                    \leq & C(1+(a^{-1}Q)^{\frac{p-1}{p}} J_{p,2} ^{\frac{s-(a+1)(p-1)}{sp}}+(a^{-1}Q)^{\frac{q-1}{q}}J_{q,2}^{\frac{s-(a+1)(q-1)}{sq}}+J_{t,2}^{\frac{s-1}{s}}).
                \end{align}
               By the estimate of $J_{p,2}, J_{q,2}$ and $Q$, we have
               \begin{align}
                   &\limsup_{n\xrightarrow{}\infty}J_{p,2}\leq Ca^{\frac{(p-1)s}{s-(a+1)(p-1)}},
                   \\
                   &\limsup_{n\xrightarrow{}\infty}J_{q,2}\leq Ca^{\frac{(q-1)s}{s-(a+1)(q-1)}},
                   \\
                   &\limsup_{n\xrightarrow{}\infty} Q\leq C,
                   \\
                   &\limsup_{n\xrightarrow{}\infty} J_{t,2}\leq C,
               \end{align}
               for $a$ small enough. Thus, we have
               \begin{align*}
                   \limsup_{n\xrightarrow{}\infty}(1+(a^{-1}Q)^{\frac{p-1}{p}} J_{p,2} ^{\frac{s-(a+1)(p-1)}{sp}}+(a^{-1}Q)^{\frac{q-1}{q}}J_{q,2}^{\frac{s-(a+1)(q-1)}{sq}}+J_{t,2}^{\frac{s-1}{s}})\leq C.
               \end{align*}
               Upon on (4.10) and letting $n\xrightarrow{}\infty$, we obtain
               \begin{align*}
                   1+\int_0^\infty\int_MVu^s\varphi^bd\mu dt\leq C,
               \end{align*}
               that is,
               \begin{align*}
                   \int_0^\infty\int_MVu^s\varphi^bd\mu dt\leq C.
               \end{align*}
               From (4.8), (4.11)-(4.13) and letting $n\xrightarrow{}\infty$ again, by a similar argument, we deduce
               \begin{align*}
                   &\int\int_{E_R}Vu^sd\mu dt
                   \\
                   \leq &C\left(\int\int_{S\setminus E_R}Vu^s\varphi^bd\mu dt\right)^{\frac{(a+1)(p-1)}{sp}}+C\left(\int \int_{S\setminus E_R}Vu^s\varphi_n^bd\mu dt \right)^{\frac{(a+1)(q-1)}{sq}}
                   \\
                   & +C \left(\int \int_{S\setminus E_R}Vu^s\varphi_n^bd\mu dt \right)^{\frac{1}{s}}.
               \end{align*}
               Passing to the limit as $R \xrightarrow{}\infty$, we conclude that
               \begin{align*}
                   \int_0^\infty\int_{M}Vu^sd\mu dt=0,
               \end{align*}
               and thus $u=0$ a.e in $S$. {\qed}

               \section{Proof of Corollaries}
               Since the proof is similar, we only proof Corollary 1.9. 

               ~\\
         $\mathit{Proof\ of\ Corollary\ 1.9.}$ We only need to show that HP2 holds under the condition in Corollary 1.9. Firstly, for small $\epsilon>0$, we have
         \begin{align*}
             &\int\int_{E_{2^{1/\theta_2}R}\setminus  E_{R}} t^{(\theta_1-1)(\frac{s}{s-1}-\epsilon)} V^{-\bar{s}_2+\epsilon}d\mu dt 
             \\
             \leq & CR^{\frac{\theta_2}{\theta_1}(\theta_1-1)(\frac{s}{s-1}-\epsilon)}\int\int_{E_{2^{1/\theta_2}R}\setminus  E_{R}}  (f(t)h(x))^{-\bar{s}_2+\epsilon}d\mu dt
             \\
             \leq & CR^{\frac{\theta_2}{\theta_1}(\theta_1-1)(\frac{s}{s-1}-\epsilon)+\frac{\theta_2}{\theta_1}\alpha\epsilon+\beta\epsilon+\frac{\theta_2}{\theta_1}\sigma_2+\sigma_1} (\mathrm{log}\ R)^{\delta_1+\delta_2}.
         \end{align*}
         Hence, (1.9) holds if we choose a large $C_0$ and
         \begin{align}
             \frac{\theta_2}{\theta_1}(\sigma_2-\frac{s}{s-1})+\sigma_1\leq 0,\quad \delta_1+\delta_2\leq \bar{s}_2.
         \end{align}
         For $z\in \{p,q\}$, large enough $R$ and small enough $\epsilon$, we find that
         \begin{align*}
             &\int\int_{E_{2^{1/\theta_2}R}\setminus  E_{R}} r(x)^{(\theta_2-1)z(\frac{s}{s-z+1}-\epsilon)} V^{-\bar{s}_{4,z}+\epsilon}d\mu dt \\
             \leq& C R^{(\theta_2-1)z(\frac{s}{s-z+1}-\epsilon)}\int\int_{E_{2^{1/\theta_2}R}\setminus  E_{R}} (f(t)h(x))^{-\bar{s}_{4,z}+\epsilon}d\mu dt
             \\
             \leq & C R^{(\theta_2-1)z(\frac{s}{s-z+1}-\epsilon) +\frac{\theta_2}{\theta_1}\alpha\epsilon+\beta\epsilon+\frac{\theta_2}{\theta_1}\sigma_{4,z}+\sigma_{3,z}} (\mathrm{log}\ R)^{\delta_{3,z}+\delta_{4,z}}.
         \end{align*}
         Similarly, we have
         \begin{align*}
             &\int\int_{E_{2^{1/\theta_2}R}\setminus  E_{R}} r(x)^{(\theta_2-1)z(\frac{s}{s-z+1}+\epsilon)} V^{-\bar{s}_{4,z}-\epsilon}d\mu dt 
             \\
             \leq & C R^{(\theta_2-1)z(\frac{s}{s-z+1}+\epsilon)+\frac{\theta_2}{\theta_1}\alpha\epsilon+\beta\epsilon +\frac{\theta_2}{\theta_1}\sigma_{4,z}+\sigma_{3,z}} (\mathrm{log}\ R)^{\delta_{3,z}+\delta_{4,z}}.
         \end{align*}
         From the above estimates, (1.10) and (1.11) hold if we choose a large $C_0$ and
         \begin{align}
             \sigma_{3,z}-\frac{sz}{s-z+1}+\frac{\theta_2}{\theta_1}\sigma
             _{4,z}\leq 0,\quad \delta_{3,z}+\delta_{4,z}\leq \bar{s}_{4,z}.
         \end{align}
         Under the condition of Corollary 1.9,  we can choose suitable $\theta_2,\theta_1$ such that
         \begin{align*}
             \sigma_1\left(\frac{s}{s-1}-\sigma_2 \right)^{-1}\leq \frac{\theta_2}{\theta_1}\leq \sigma_{4,z}^{-1} ( \frac{sz}{s-z+1}-\sigma_{3,z})
         \end{align*}
         if 
         \begin{align*}
             \sigma_{4,z}>0,\quad \frac{s}{s-1}-\sigma_2>0,\quad \frac{sz}{s-z+1}-\sigma_{3,z}>0.
         \end{align*}
         This completes the proof. {\qed}

       \section{ Acknowledgments}
        The author would like to thank Professor Yuxin Dong and Professor Xiaohua Zhu for their continued support and encouragement.

\bibliographystyle{siam}
\bibliography{ref}

\end{document}